\documentclass[11pt]{amsart}
\pdfoutput=1
\usepackage[left=3cm, right=3cm, top=3cm, bottom=3cm]{geometry}

\usepackage{amsmath,amsfonts,amssymb,amsthm,xfrac,subcaption}
\usepackage[utf8]{inputenc}
\usepackage[english]{babel}

\newtheorem{theorem}{Theorem}[section]

\newtheorem{definition}[theorem]{Definition}
\newtheorem{lemma}[theorem]{Lemma}

\newcommand{\norm}[1]{\left\lVert{#1}\right\rVert}
\newcommand{\N}{\mathbb{N}}
\newcommand{\R}{\mathbb{R}}
\newcommand{\C}{\mathbb{C}}
\newcommand{\ip}[1]{\left\langle {#1} \right\rangle}

\newcommand{\setbar}{{\hspace{0.1cm}\big{|}\hspace{0.1cm}}}

\newcommand{\spn}{{\text{span}}}

\usepackage[colorlinks=true,linkcolor=red]{hyperref}
\usepackage[square,numbers]{natbib}
\usepackage{enumitem}
\begin{document}
\title{Topics in Hilbert Spaces, Spectral Theory, and Harmonic Analysis}
\author{Sawyer Jack Robertson}
\date{\today}
\vspace*{-1.2cm}
\maketitle
\tableofcontents

\section*{Author's Note}
This expository paper is based on work done while I was student at the University of Oklahoma under the guidance of Prof. Tomasz Przebinda in Spring 2018 and Spring 2019. In the time since, it has sat in my archive gathering dust and I have decided this year to return to this document, polish it, and share it with the community. This paper doesn't contain original research and is not intended for publication. Many thanks to Prof. Przebinda for his questions which have motivated my study of these various topics, and his corrections and suggestions as I have written this paper. \par
The contents of this paper are wide-ranging but connected by the topics of Hilbert spaces, spectral theory, and abstract harmonic analysis. The preliminary section includes the requisite facts and theorems for the remainder of the paper, as well as a digression on Riesz bases. The second and third sections are focused on deriving the general spectral theorem for Hilbert spaces in finite and arbitrary dimension, respectively. The fourth and fifth sections present applications of the general theory: first to finite abelian groups, then to the space $L^2(S^1,\mu)$.

\newpage
\section{Banach Algebras and Hilbert Spaces}

This preliminary section will consist of the fundamental definitions and tools of Banach algebras and Hilbert spaces, followed by a brief digression on Riesz bases for Hilbert spaces. Bold is occasionally in this section to help the reader refer back to these definitions quickly. $V$ will be a vector space. Unless otherwise specified, the field for $V$ will be taken as $\C$.
    
\subsection{Banach Algebras}

Shortly we will consider the spectral theory of certain classes of operator algebras. We will write down some preliminary details here so that the discussion may proceed more smoothly later on. This covers the general theory of Banach Algebras.

\begin{definition}
    Let $V$ be a vector space. A map $\norm{\cdot}:V\rightarrow\R$ is called a norm provided that for each $x,y\in V$ and $\alpha\in\C$,
            \begin{enumerate}
                \item $\norm{x}\geq 0$,
                \item $\norm{x+y}\leq \norm{x}+\norm{y}$,
                \item $\norm{\alpha x}=|\alpha|\norm{x}$,
                \item $\norm{x}=0\iff x= 0$.
            \end{enumerate}
    $(V,\norm{\cdot})$ is a called a \textbf{normed space}.
\end{definition}

Examples of normed spaces include $\C^n$ with any one of the following choices:
    \begin{enumerate}[label=\textit{(\roman*)}]
        \item $\norm{(z_1,z_2,\dots,z_n)}_2^2:=\sum_{i=1}^n|z_i|^2$
        \item $\norm{(z_1,z_2,\dots,z_n)}_\infty:=\max_{1\leq i\leq n}|z_i|$
        \item $\norm{(z_1,z_2,\dots,z_n)}_1:=\sum_{i=1}^n|z_i|$.
    \end{enumerate}

    \begin{definition}
        Let $\{x_n\}_{n=0}^\infty$ be a sequence in a normed space $V$. Then $x_n$ is said to be Cauchy if for each $\epsilon>0$ there is a natural number $N\geq 0$ such that for any $n,m\geq N$ it holds
            \[\norm{x_n-x_m}<\epsilon.\]
        A sequence $\{x_n\}_{n=0}^\infty$ in a normed space $V$ is said to converge to a limit $x\in V$ if for each $\epsilon>0$ there is a natural number $N\geq 0$ such that for any $n\geq N$, it holds
            \[\norm{x_n-x}<\epsilon.\]
        Finally, if each Cauchy sequence in a normed space $V$ is indeed convergent to a limit in $V$, then $V$ is called complete. In this important case, $V$ is called a \textbf{Banach space}.
    \end{definition}

\begin{definition}
    An \textbf{algebra} $\mathcal{A}$ over a field $\mathbb{F}$ is a vector space equipped with a bilinear product; that is, a mapping $\otimes:\mathcal{A}\times \mathcal{A}\rightarrow \mathcal{A}$ satisfying the following:
        \begin{enumerate}
            \item $(x+y)\otimes z=x\otimes x+y\otimes z$,
            \item $x\otimes (y+z)=x\otimes y+x\otimes z$,
            \item $ (\alpha\beta)x\otimes y=(\alpha x)\otimes (\beta y)$,
        \end{enumerate}
    for each $x,y,z\in\mathcal{A}$ and $\alpha,\beta\in\mathbb{F}$. $\mathcal{A}$ is called unital if there exists an identity element $e\in\mathcal{A}$ for which $e\otimes x=x\otimes e=x$ for each $x\in\mathcal{A}$. An element $x\in\mathcal{A}$ is called invertible if there exists an element $x^{-1}$ satisfying $x^{-1}\otimes x=x\otimes x^{-1}=e$. To be concise, the product operation is henceforth denoted by concatenation. 
\end{definition}

Algebras include $\R,\C$, as well as $M_n(\C)$, the space of $n\times n$ complex matrices. 

\begin{definition}
    Let $\mathcal{A}$ be an algebra and suppose $\mathcal{A}$ is equipped with a norm structure $\norm{\cdot}$ with respect to which $\mathcal{A}$ forms a Banach space. $\mathcal{A}$ is called a \textbf{Banach algebra} if, for each $x,y\in\mathcal{A}$ it holds
        \[\norm{xy}\leq\norm{x}\norm{y}.\]
\end{definition}

We introduce an important concept to pretext an example of a Banach algebra. Suppose $V,W$ are two vector spaces spaces over the same field, and $T:V\rightarrow W$ is a linear mapping. We say $T$ is \textit{bounded} if there exists a $C\geq 0$ for which $\norm{Tx}\leq C\norm{x}$ for each $x\in V$. It can be easily proved that the $\epsilon$-$\delta$ definition of continuity will coincide with the notion of boundedness.

\begin{definition}
    Suppose $V,W$ are vector spaces. We define the operator space $\mathcal{L}(V,W)$ to be the set of bounded operators $T:V\rightarrow W$. In a specific case, $\mathcal{L}(V)$ is the space of linear operators $V\rightarrow V$.
\end{definition}

\begin{definition}
    Let $V$ be a normed space. Then the following map $\norm{\cdot}$ defined by
        \[\norm{T}:=\inf\{C\geq 0\setbar \norm{Tx}\leq C\norm{x}\text{ for each $x\in V$}\}\]
    is a norm on $\mathcal{L}(V)$. Moreover, $\norm{\cdot}$ satisfies, for each $S,T\in\mathcal{L}(V)$,
        \[\norm{ST}\leq\norm{S}\norm{T}.\]
    If $V$ is a Banach space, then $\mathcal{L}(V)$ is a Banach algebra.
\end{definition}

\begin{theorem}[Gelfand-Mazur]
    Suppose $\mathcal{A}$ is a Banach algebra for which each $x\in\mathcal{A}\backslash\{0\}$ is invertible. Then $\mathcal{A}$ is isomorphic to $\C$, in the sense that there is a mapping $\phi:\mathcal{A}\rightarrow\C$ which is an isometric isomorphism of vector spaces.
\end{theorem}

The proof of this fact, while relatively accessible, does require a fair bit of machinery which we prefer to defer to the literature. Good sources on this include Folland\cite[p. 4]{Fo15} and Remling\cite[p. 73]{Re08}.

    \begin{definition}
        Let $\mathcal{A}$ be a commutative Banach algebra. Then a subspace $K\subset A$ is called a right ideal if for each $x\in K$ and $y\in \mathcal{A}$, $xy\in K$.
    \end{definition}

We say a right ideal $K$ is \textit{proper} if $K\neq \mathcal{A}$, and we notice that such an ideal is proper if and only if it does not contain any invertible elements (for if $x\in K$ is invertible, then $xx^{-1}=e\in K$ whence $K=\mathcal{A}$). Moreover, we say a proper ideal is \textit{maximal} if it is not contained in any larger proper ideals. Note every proper ideal is contained in some maximal ideal. This motivates the following lemma:

    \begin{lemma}\label{idealfunct}
        Let $\mathcal{A}$ be a commutative Banach algebra and suppose $K$ is a maximal ideal in $\mathcal{A}$. Then $K$ may be realized as the kernel of a functional $\psi:\mathcal{A}\rightarrow\C$ which is a multiplicative homomorphism.
    \end{lemma}

    \begin{proof}
        Put a relation on $\mathcal{A}$ by setting $x\sim y$ if $x-y\in K$. Let $\pi:\mathcal{A}\rightarrow \mathcal{A}\backslash K:x\mapsto[x]$ to be the quotient mapping sending each $x$ to its equivalence class in the quotient space. Using the standard quotient norm
            \[\norm{[x]}_{\mathcal{A}\backslash K}:=\inf_{k\in K}\norm{x-k}_\mathcal{A},\]
        the quotient space $\mathcal{A}\backslash K$ becomes a Banach algebra (note: multiplication in the quotient is well-defined from the commutativity of the algebra). Moreover, every nonzero element of this quotient space is invertible since the quotient space does not contain any proper ideals; whence, by Gelfand Mazur theorem, it is isomorphic to $\C$ by some isomorphism $\phi$. Set $\psi=\phi\circ\pi$; then, $\psi$ is a multiplicative functional on $\mathcal{A}$, taking the value 0 strictly on elements of $K$ as desired.
    \end{proof}

    \begin{definition}
        An \textbf{involution} on a Banach algebra $\mathcal{A}$ is a map $x\mapsto x^\ast$ so that for each $x,y\in \mathcal{A}$ and $\lambda\in\C$, the following hold:
            \begin{enumerate}
                \item $(x+y)^\ast=x^\ast+y^\ast$
                \item $(\lambda x)^\ast = \overline{\lambda}x^\ast$
                \item $(xy)^\ast = y^\ast x^\ast$
                \item $x^{\ast\ast}=x$
        \end{enumerate}
    \end{definition}
    
    \begin{definition}
        A Banach algebra $\mathcal{A}$ equipped with an involution $\ast$ is called a $\ast$-algebra; if $\norm{xx^\ast}=\norm{x}^2$ for each $x\in \mathcal{A}$, then $\mathcal{A}$ is called a $\mathbf{C}^\ast$\textbf{-algebra}.
    \end{definition}

%%%%%%%%%%%%%%%%%%%%%%%%%%%%%%%%%%%%%%%%%%%%%%%%%%%%%%%%%%%
\subsection{Hilbert Spaces}

    \begin{definition}
    Let $V$ be a vector space. A \textit{Hermitian inner product} on $V$ is a map $\langle\cdot,\cdot\rangle:V\times V\rightarrow\C$ satisfying the following four axioms for each $x,y,z\in V$ and $\alpha\in\C$: 
    
        \[\begin{array}{lc}
            \text{(Linearity in first argument)} & \langle x+y,z\rangle=\langle x,z\rangle+\langle y,z\rangle \\ 
            \text{(Homogeneity in first argument)} & \langle \alpha x, y\rangle=\alpha\langle x,y\rangle \\
            \text{(Conjugate symmetry)} & \langle x,y\rangle=\overline{\langle y,x \rangle}\\
            \text{(Positive definiteness)} & \langle x,x\rangle\geq 0\text{ and }\langle x,x\rangle =0\iff x=0
        \end{array}\]
    where the overline in the thir axiom denoted complex conjugation.
    \end{definition}

We say two vectors $x,y\in V$ are \textit{orthogonal} if $\ip{x,y}=0$. We define the \textit{induced norm} or \textit{length} of a vector $x\in V$ to be $\norm{x}=\sqrt{\ip{x,x}}$. A vector space equipped with a Hermitian inner product is called an inner product space.

    \begin{definition}
    Let $V$ be an inner product space. An \textbf{orthonormal basis} for $V$ is a collection of vectors $\{u_i\}_{i\in I}$ for which the following hold:
        \begin{enumerate}
            \item Each $u_i$ has norm one,
            \item each pair of distinct vecotrs in the collection are orthogonal,
            \item each $v\in V$ has a representation $v=\sum_{i\in I}v_iu_i$ as a converging series.
        \end{enumerate}
    \end{definition}
    
    \begin{definition}
        If $V$ is a complete inner product space, $V$ is called a \textbf{Hilbert space}.
    \end{definition}
    
One important example of a Hilbert space is $\C^n$ under the inner product
    \[\ip{(z_1,z_2,\dots,z_n),(w_1,w_2,\dots,w_n)}:=\sum_{i=1}^nz_i\overline{w_i}.\]
Another key example that will show up in another form later is the space $L^2(\R,\mu)$ of square-integrable functions:
    \[L^2(\R,\mu):=\left\{f:\R\rightarrow\C\text{ measurable}\setbar\int_\R|f(x)|^2d\mu(x)<\infty\right\}\]
where $\int d\mu$ is the standard Lebesgue measure on $\R$, see e.g., Royden \& Fitzpatrick \cite[Part 1]{Ro10}. The inner product in this space, similar to the preceding example, is
    \[\ip{f,g}:=\int_\R f(x)\overline{g(x)}d\mu(x).\]
Before moving onto operators, some quick topology. A set $U\subset V$ in a Hilbert space $V$ is said to be \textit{closed} if it contains the limits of all of the convergent sequences in the space. A set $O\subset V$ is said to be \textit{open} if for each $x\in O$ there is an $\epsilon>0$ for which
    \[\{y\in V\setbar \norm{x-y}<\epsilon\}\subset O,\]
that is, $O$ contains an open ball around the element with some positive radius. A Hilbert space $V$ is said to be \textit{separable} if it contains a sequence $\{x_n\}_{n=0}^\infty$ so that for any nonempty set $O\subset V$, there is an element of the sequence $\{x_n\}_{n=0}^\infty\ni x_k\in O$.
    \begin{theorem}
        A Hilbert space $V$ contains a countable orthonormal basis if and only if it is separable.
    \end{theorem}
We defer the proof of this theorem to the literature.\par
Now we transition away from the algebraic and analytical structure of a Hilbert space and focus for a moment on the analytical structures associated to linear operators, in particular, those on Hilbert spaces. 

    \begin{definition}
        We define the \textbf{dual space} of $V$, denoted $V^\ast$ to be the vector space of continuous complex-valued linear transformations, or the set of linear functionals on $V$. 
    \end{definition}

In the case where $V$ has finite dimension, the continuity requirement is redundant. In the infinite-dimensional case, it is essential. 

    \begin{definition}
        Let $V$ be a Hilbert Space, and suppose $T\in\mathcal{L}(V)$. We define the \textbf{adjoint} of $T$ to be the unique operator $T^\ast$ satisfying
            \[\ip{Tx,y}=\ip{x,T^\ast y}\]
        for each $x,y\in V$.
    \end{definition}
    
The existence of an adjoint operator for each $T\in\mathcal{L}(V)$ is a consequence of the Riesz representation theorem for Hilbert Spaces. In fact, a matrix representation of an adjoint to an operator on a finite-dimensional Hilbert Space can be found by taking the conjugate transpose of a matrix representation of the operator. We say an operator is \textbf{\textit{normal}} if it commutes with its adjoint, and we say an operator is \textit{self-adjoint} if it is equal to its adjoint. 

    \begin{theorem}
    Let $V$ be a finite dimensional Hilbert Space, and suppose  $T\in\mathcal{L}(V)$ is normal. Then for each $x\in V$,
    \[\norm{Tx}=\norm{T^\ast x}\]
    \end{theorem}
    
    \begin{proof}
    The proof is routine: Fix $x\in V$, and notice
    \[\ip{Tx,Tx}=\ip{x,T^\ast Tx}=\ip{x,TT^\ast x}=\ip{T^\ast x,T^\ast x}\]
    \end{proof}
We now give some spectral theoretic definitions and two lemmas which will be used later on as needed.

    \begin{definition}
    Suppose $T\in\mathcal{L}(V)$ for some Hilbert Space $V$ has eigenvalue $\lambda\neq 0$. We define the eigenspace associated to $\lambda$ to be
    \[\mathcal{E}_\lambda:=\{v:Tv=\lambda v\}.\]
    \end{definition}

    \begin{lemma}
    Let $V$ be some Hilbert Space, and let $T\in\mathcal{L}(V)$ be normal. Then for every nontrivial $v\in V$, we have $Tv=\lambda v\iff T^\ast v=\overline{\lambda}v$.
    \end{lemma}
    
    \begin{proof}
    Since $T,T^\ast$ commute, it is clear that for each $v,w\in V$ we have
    \[\ip{Tv,Tw}=\ip{T^\ast v,T^\ast w}.\]
    whence $\ker{T}=\ker{T^\ast}$. Making use of the fact that $T-\lambda I$ is also normal, we see
    \[Tv=\lambda v\iff v\in \ker{(T-\lambda I)}\iff v\in \ker{(T^\ast -\overline{\lambda}I)}\iff T^\ast v=\overline{\lambda} v.\]
    \end{proof}

Three more useful definitions will round out this section.

    \begin{theorem}
        Let $V$ be a separable Hilbert Space, and let $U\subset V$ be a closed linear subspace of $V$. Define the projection onto $U$, denoted $P_U$, by $P_Uv=\sum_i\ip{u_i,v}u_i$ for some fixed orthonormal basis $\{u_i\}$ of $U$. Then $P_U$ satisfies:
            \begin{enumerate}
                \item $P_U^2=P$,
                \item $P_U^\ast=P_U$,
                \item $\operatorname{Range}P_U=U$,
                \item For closed linear subspaces $U,W\subset V$, $P_{U\cap W}=P_UP_W=P_WP_U$.
            \end{enumerate}
    \end{theorem}
    
    The proof of this is deferred to the reader; the first three are almost immediate computations. The fourth can be done through computation again; it might be helpful to first find an orthonormal basis for the intersection, extend it to two separate bases for $U,W$, and then compute. One detail that has been somewhat overlooked is the independence of the defintion of $P_U$ on the choice of orthonormal basis used to define it explicitly (this is a cleverer but still striaghtforward computation).

    \begin{definition}
        Let $V$ be a finite-dimensional vector space, and let $T\in\mathcal{L}(V)$. We say that $T$ is diagonalizable with respect to a specified basis of $V$ if its matrix representation, taken with respect this basis, is diagonal. 
    \end{definition}
    
    \begin{definition}
        Let $V$ be a Hilbert space, and let $\{W_\lambda\}_{\lambda\in\Lambda}$ be some family of mutually orthogonal subspaces of $V$. We say that $V$ is the orthogonal sum of $\{W_\lambda\}_{\lambda\in\Lambda}$ and write
            \[V=\bigoplus_{\lambda\in\Lambda}W_\lambda\]
        if each $w\in V$ has a unique representation as a converging series $w=\sum_{\lambda\in\Lambda}c_\lambda w_\lambda$ of vectors in $w_\lambda\in W_\lambda$.
    \end{definition}
    
    A trivial example of an orthogonal decomposition of, say, a separable Hilbert space $V$ would be as the orthogonal sum of the spaces generated by each vector in its countable orthonormal basis.

%%%%%%%%%%%%%%%%%%%%%%%%%%%%%%%%%%%%%%%%%%%%%%%%%%%%%%%%%%

\subsection{Digression: Riesz bases}

Riesz bases have arisen in the modern theory of Hilbert spaces as a cousin to the classical orthonormal basis, but which are somewhat weaker in their structure. This have been used, for example, in the construction of multiresolution approximations of $L^2(\R)$, important in the theory of wavelets and harmonic analysis. This topic is somewhat orthogonal to the overall discussion, but fits nicely within a discursive treatment of Hilbert spaces. The purpose of this section will be to establish three classical characterizations of Riesz bases for a given Hilbert space. Most of the arguments here are adapted from \cite[Ch. 3]{OC}.\par

\begin{definition}
    A sequence $\{x_n\}_{n=1}^\infty$ in a Hilbert space ${V}$ is called complete if it is linearly independent, and if for each $\epsilon>0$ and $x\in{V}$ there is a finite linear combination $\sum_{j}c_j x_{n_j}$ for which $\norm{x-\sum_{j}c_j x_{n_j}}<\epsilon$.
\end{definition}

\begin{definition}\label{bessel}
    Let ${V}$ be a Hilbert space. A sequence $\{x_n\}_{n=1}^\infty\subset{V}$ is a Bessel sequence if there is a constant $B>0$ such that for each $x\in{V}$ it holds 
        \[\sum_{n=1}^\infty |\ip{x,x_n}|^2\leq B\norm{x}^2.\]
    The least such $B$ for which the inequality holds will be called the Bessel constant of the sequence.
\end{definition}

\begin{lemma}\label{rieszcomplete}
Let ${V},{W}$ be Hilbert spaces, with $\{x_n\}_{n=1}^\infty\subset{V}$ and $\{y_n\}_{n=1}^\infty\subset{W}$; where $\{x_n\}_{n=1}^\infty$ is a Bessel sequence with constant $B$, and $\{y_n\}_{n=1}^\infty$ is complete in ${V}$. Moreover assume that there exists $A>0$ for which 
    \[A\sum_{n=1}^N |c_n|^2\leq\norm{\sum_{n=1}^N c_ny_n}^2\]
for any finite sequence of scalars $\{c_n\}_{n=1}^N$. Then
    \[U(\sum_{n=1}^N c_n y_n):=\sum_{n=1}^N c_n x_n\]
defines a bounded map from $\spn\{y_n\}_{n=1}^\infty\rightarrow\spn\{x_n\}_{n=1}^\infty$ which has a unique bounded extension from ${V}\rightarrow{W}$ whose norm is at most $\sqrt{B/A}$.
\end{lemma}

\begin{proof}
Since $\{y_n\}_{n=1}^\infty$ is complete in ${V}$, any element of its span is uniquely represented as a finite linear combination of its elements, which confirms the well-definition of the operator $U$. Given a finite sequence of scalars $\{c_n\}_{n=1}^N$, since $\{x_n\}_{n=1}^\infty$ is Bessel, we have the estimate
    \[\begin{split}
        \norm{U(\sum_{n=1}^N c_n y_n)}^2&=\norm{\sum_{n=1}^N c_n x_n}^2\leq B\sum_{n=1}^N |c_n|^2\\
            &\leq \frac{B}{A}\norm{\sum_{n=1}^N c_n y_n}^2
    \end{split}\]
whence $U$, defined between the spans, is bounded. Since $\{y_n\}_{n=1}^\infty$ is dense, we can extend $U$ to the entire space ${V}$ whilst preserving the norm estimate above.
\end{proof}

\begin{theorem} \label{rieszbasis}
    Let ${V}$ be a Hilbert space, and $\{x_n\}_{n=1}^\infty\subset{V}$ be a sequence. Then the following are equivalent:
        \begin{enumerate}[label=(\roman*)]
            \item $x_n=Ue_n$ for each $n\geq 1$, where $\{e_n\}_{n=1}^\infty\subset{V}$ is some orthonormal basis for ${V}$ and $U$ is a bounded, bijective operator on ${V}$.
            \item $\{x_n\}_{n=1}^\infty$ is complete in ${V}$ and there exist scalars $A,B>0$ such that for every finite sequence of scalars $\{c_n\}_{n=1}^ N$, it holds
                \[A\sum_{n=1}^N |c_n|^2\leq \norm{\sum_{n=1}^N c_n x_n}^2\leq B\sum_{n=1}^N |c_n|^2.\]
            \item $\{x_n\}_{n=1}^\infty$ is complete in ${V}$ and its Gram matrix $a_{ij}:=\ip{x_i,x_j}$ defines a bounded, invertible linear operator on $\ell^2(\mathbb{N})$ (Hilbert space of square-summable complex sequences).
        \end{enumerate}
\end{theorem}

\begin{proof} $\left(i\Rightarrow ii\right)$. Let $x_n=Ue_n$ for each $n\geq 1$ and some aforementioned $U,\{e_n\}_{n=1}^\infty$; we estimate
    \[\norm{\sum_{n=1}^N c_n x_n}^2\leq \norm{U\left(\sum_{n=1}^N c_n e_n\right)}^2\leq\norm{U}^2\norm{\sum_{n=1}^N c_n e_n}^2=\norm{U}^2\sum_{n=1}^N|c_n|^2.\]
Similarly,
    \[\sum_{n=1}^N|c_n|^2=\norm{\sum_{n=1}^N c_n e_n}^2=\norm{U^{-1 }U\left(\sum_{n=1}^N c_n e_n\right)}^2\leq \norm{U^{-1}}^2\norm{\sum_{n=1}^N c_n x_n}^2.\]
Putting the two estimates together gives the first implication.\par

$\left(ii\Rightarrow i\right)$. Assume for a moment that we are able to deduce from the inequality $(ii)$ that $\{x_n\}_{n=1}^\infty$ is itself a Bessel sequence. We can then fix an orthonormal basis $\{e_n\}_{n=1}^\infty$ and define an operator $U:e_n\mapsto x_n$ on ${V}$ whose boundedness would follow from Lemma \ref{rieszcomplete} (here, `$A$' as in the lemma is 1 since $\{e_n\}_{n=1}^\infty$ is an orthonormal basis). Similarly, the operator $V:x_n\mapsto e_n$ on ${V}$ would be bounded by the lemma (here, `$A$' as in the lemma is the same `$A$' as in $(ii)$); since $UV=VU=\text{Id}$, we have the desired representation of our sequence $\{x_n\}_{n=1}^\infty=\{Ue_n\}$. Returning to the initial comment we need to prove that $\{x_n\}_{n=1}^\infty$ is Bessel, for which I found an interesting (if roundabout) argument. Recall for a moment that if $F:X\rightarrow Y$ is a bounded operator between normed spaces, we can write its formal transpose operator via
    \[F^t:Y^\ast\rightarrow X^\ast:y^\ast\mapsto F^ty^\ast,\hspace{0.5cm} \left(F^t y^\ast\right)(x)= y^\ast(Fx).\]
Looking at the R.H.S. of the assumed inequality and taking limit where needed, we can observe that the operator $T:\ell_2(\N)\rightarrow{V}$ defined by $\{c_n\}_{n=1}^\infty \mapsto \sum_{n=1}^\infty c_n x_n$ is bounded. Suppose we take a functional in ${V}^\ast$; by Riesz representation, we can write it in the form $\ip{\cdot, x}_{{V}}$ where $x\in{V}$. For $\{c_n\}_{n=1}^\infty\in\ell_2(\mathbb{N})$ fixed, we can then look at
    \[ \left(T^t\ip{\cdot,x}\right)\left(\{c_n\}_{n=1}^\infty\right)=\ip{\sum_{n=1}^\infty c_n x_n,x} =\sum_{n=1}^\infty c_n\ip{x_n,x}=\ip{\{c_n\}_{n=1}^\infty,\{\ip{x_n,x}\}_{n=1}^\infty}_{\ell^2(\mathbb{N})}.\]
So, we have proved that $T^t:\ip{\cdot, x}\mapsto\ip{\cdot,\{\ip{x_n ,x}\}}_{\ell^2(\mathbb{N})}$. Since $T$ was bounded, so is $T^t$, as well as the formal mapping $x\mapsto\{\ip{x_n,x}\}$. The inequality in Definition \ref{bessel} holds by this duality argument, and the implication is proved. One remark: though this approach isn't `sexy,' it is revealing: the upper bound in the desired Bessel condition is in some sense dual to the upper bound in (ii).\par

$\left(i\Rightarrow iii\right)$. Let $x_n=Ue_n$ for each $n\geq 1$ and some aforementioned $U,\{e_n\}_{n=1}^\infty$ as in $(i)$. Then an entry of the Gram matrix would be
    \[\ip{x_i,x_j}=\ip{Ue_i,Ue_j}=\ip{U^\ast Ue_i,e_j}\]
which is the $i,j$-th entry in the matrix representation of the bounded operator $U^\ast U$ on ${V}$ in the basis $\{e_n\}_{n=1}^\infty$, as desired.\par

$\left(iii\Rightarrow i\right)$. Now we assume that the Gram matrix $a_{ij}=\ip{x_i,x_j}_{{V}}$ defines a nice bounded operator on $\ell^2(\mathbb{N})$. Fix an orthonormal basis $\{e_n\}_{n=1}^\infty \subset{V}$ and define a new bounded operator $T$ on ${V}$ by the equation $\ip{T e_i,e_j}=\ip{x_i,x_j}$. Considering $\sum_{n=1}^\infty c_ne_n\in{V}$, we compute
    \[\ip{T\left(\sum_{n=1}^\infty c_n e_n\right),\sum_{j=1}^\infty c_j e_j}=\sum_{n,j=1}^\infty c_nc_j\ip{Te_n,e_j}=\sum_{n,j=1}^\infty c_nc_j\ip{x_n,x_j}=\norm{\sum_{n=1}^\infty c_n x_n}^2\geq 0.\]
A similar computation will show that $T$ is also self-adjoint. Since $T$ is positive and self-adjoint, by the result \cite[Thm. 5.1.3]{DE}, we can find a bounded operator $R$ on ${V}$ for which $T=R^\ast R$, so that $$\ip{x_i,x_j}=\ip{T e_i, e_j}=\ip{R e_i, R e_j}.$$
The invertibility of $R$ follows from the assumed invertibility of $\ip{x_i,x_j}$ on $\ell^2$. This completes the proof.
\end{proof}

\begin{definition}
    A sequence $\{x_n\}_{n=1}^\infty$ in a Hilbert space ${V}$ is called a Riesz basis if any of the conditions in Theorem \ref{rieszbasis} is satisfied.
\end{definition}

\section{Spectral Theorem for Finite Dimensional Hilbert Spaces}
In this section, we state and prove the spectral theorem for a finite dimensional Hilbert space. Before addressing this goal, however, we look at the Schur decompositon of an $n\times n$ complex matrix. First, some terminology and notation. These arguments follow those used by Axler\cite[§ 7B]{Ax15}; this author finds his approach particularly accessible and straightforward. \par
We use $M_n(\C)$ to denote the vector space of $n\times n$ matrices with complex entries; the identity matrix is denoted $\text{Id}$. Recall that if $A\in M_n(\C)$, we can consider its action on $\C^n$ by linear transformation and speak of its adjoint operator. One verifies that the matrix representation of the adjoint of $A$ is in fact
    \[\overline{A}^t=:A^\ast\]
where the superscript $t$ denotes matrix transposition. We use the notation
    \[U_n:=\{U\in M_n(\C)\setbar U^\ast U=\text{Id}\in M_n(\C)\}\]
for $n\times n$ unitary matrices.

    \begin{theorem}[Schur Decomposition of $M_n(\C)$]
        Let $A\in M_n(\C)$. Then there exists a $U\in{U_n}$, and an upper-triangular matrix $B$, so that
            \begin{equation}\label{schur}
                U^\ast AU=B.
            \end{equation}
    \end{theorem}

    \begin{proof}
    The theorem is vacuous in the case of $M_1(\C)$, so let us assume inductively that it holds for any $M_{k}(\C)$, $k\leq n-1$. Fix $A\in M_n(\C)$, and choose $\lambda\in\C$, ${x}\in\C^n$ with $\norm{{x}}=1$, so that $A{x}=\lambda{x}$ (existense here is due to the existence of a complex eigenvalue for any given matrix, \textit{c.f.} Fundamental Theorem of Algebra). Now choose some $V\in{U_n}$ with first column equal to ${x}$ (e.g. extend ${x}$ to some orthonormal basis of $\C^n$ via, say, Gram-Schmidt, and write in matrix form). Write $V=[{x}\hspace{0.15cm} \widetilde{V}]$. We compute:
        \begin{equation}\label{vstar}
        V^\ast A V=\begin{bmatrix}
        {x}^\ast A{x} & {x}^\ast A \widetilde{V} \\
        \widetilde{V}^\ast A {x} & \widetilde{V}^\ast A \widetilde{V}\\
        \end{bmatrix}
        \end{equation}
    We make a couple of observations. First, ${x}^\ast A{x}=\lambda {x}^\ast{x}=\lambda\norm{{x}}^2=\lambda$. Second, $\widetilde{V}^\ast A{x}=\lambda\widetilde{V}^\ast{x}=0\in\C^{n-1}$, since ${x}$ is orthogonal to the columns of $\widetilde{V}$. Then \eqref{vstar} becomes
       \[ V^\ast A V=\begin{bmatrix}
            \lambda & ... \\
            0 & \widetilde{A}\\
        \end{bmatrix}\]
    where $\widetilde{A}= \widetilde{V}^\ast A \widetilde{V}$. By our induction assumption, we may find a matrix $W\in U_{n-1}$ and an upper triangular matrix $C\in M_{n-1}(\C)$ so that $\widetilde{A}={W}^\ast C{W}$. We then may factor out the decomposition of $\widetilde{A}$ and have
        \[ V^\ast A V = 
            \begin{bmatrix}
               1 & 0\\   
               0 & W^\ast\\
            \end{bmatrix}
            \begin{bmatrix}
               \lambda & ...\\   
               0 & C\\
            \end{bmatrix}            
             \begin{bmatrix}
               1 & 0\\   
               0 & W\\
            \end{bmatrix}\]
    By setting $U=V            \begin{bmatrix}
               1 & 0\\   
               0 & W^\ast\\
            \end{bmatrix}\in{U_n}$, and $B=            \begin{bmatrix}
               \lambda & ...\\   
               0 & C\\
            \end{bmatrix}$, we arrive at \eqref{schur}.
    \end{proof}

As a corollary, notice that since $U^\ast$ is the inverse of $U$, the conjugation operation performed by $U$ on $A$ on the L.H.S. of \eqref{schur} amounts to change of basis in some sense. In other words, we get the following:

    \begin{theorem}[Schur's Theorem]
    Let $V$ be an $n$-dimensional Hilbert space, and let $T\in\mathcal{L}(V)$. Then there exists an upper-triangular matrix representation of $T$ with respect to some orthonormal basis of $V$.
    \end{theorem}

    \begin{proof}
    Fix $T\in\mathcal{L}(V)$, and choose an arbitrary matrix representation of $T$, say $S\in M_n(\C)$. Use Schur decomposition to write
    \[Q = P S P^{-1}\]
    for some $P\in{U_n}$, and some upper-triangular matrix $Q$. Since $P\in{U_n}$, its columns correspond to an orthonormal basis of $\C^n$. Notice that $Q$ is now the matrix representation of $T$ with respect to the orthonormal basis given in $P$, as desired. 
    \end{proof}

Now the first spectral theorem.

    \begin{theorem}[Spectral Theorem I]
    Let $V$ be a finite-dimensional Hilbert space, and let $T\in\mathcal{L}(V)$. Then the following are equivalent:
        \begin{enumerate}
            \item $V$ has an orthonormal basis of eigenvectors of $T$.
            \item $T$ is diagonalizable with respect to some orthonormal basis of $V$.
            \item $T$ is normal, i.e. $TT^\ast=T^\ast T$, where $T^\ast$ is the adjoint of $T$.
        \end{enumerate}
    \end{theorem}

    \begin{proof}
    We prove via the following two equivalences:
    \[ (1)\iff (2)\iff (3).\]
    If $(2)$ holds, $(1)$ is clear if one chooses the orthonormal basis in question. Conversely, if $(1)$ holds, the matrix representation of $T$ with respect to this basis must also be diagonal, so $(2)$ follows. To show $(3)$ is equivalent to $(2)$, first assume $TT^\ast=T^\ast T$. Apply Schur's Theorem to obtain an upper-triangular matrix representation of $T$ with respect to some o.n.b. of $V$; say $(T)_{ij}=\{t_{ij}\}_{i,j=1}^n$, where $t_{ij}=0$ when $i>j$. Write down the standard basis of $V$ in the form $\{e_i\}$. We recall from Theorem 1 that $\norm{T^\ast x}=\norm{T x}$ for each $x\in V$. Recalling that $(T^\ast)_{ij}=\overline{(T)_{ji}}$,
        \[|t_{1,1}|^2=\norm{Te_1}^2=\norm{T^\ast e_1}^2=\sum_{j=1}^n|t_{1,j}|^2\]
    which forces $|t_{1,j}|=0$ for $j>1$. We make a similar comparison
    using $e_2$ as follows:
        \[|t_{2,2}|^2=\norm{Te_2}^2=\norm{T^\ast e_2}^2=\sum_{j=2}^n|t_{2,j}|^2\]
    forcing $|t_{2,j}|=0$ for $j>2$. Using induction on $1\leq i\leq n$, we find that $t_{ij}=0$ for each $j>i$. Since $(T)_{ij}$ is upper triangular, it must be diagonal and $(2)$ follows. Next, we assume $(2)$. Let $(T)_{ij}$ be such a diagonal matrix representation. Then, $(T^\ast)_{ij}=\overline{(T)_{ji}}$ is also a diagonal matrix. All diagonal matrices with entries in $\C$ commute, and by extending this from matrices back to the operators, it follows that $T$ and $T^\ast$ commute. 
    \end{proof}

\section{C*-algebras, Gelfand Theory, and Spectral Theorem }\label{Cstaralgebras}

In this section, we approach the same spectral theoretic goals we developed for finite-dimensional Hilbert spaces in the preceding, but from the loftier perspective of spectral theory for operator algebras. The conclusions and techniques, underwritten by abstraction, provide similar results as we saw in the last section, but with much broader scope. We generally follow Folland's approach\cite{Fo15}.

    \begin{definition}
        Let ${\mathcal{A}}$ be a commutative unital Banach algebra, and $x\in {\mathcal{A}}$. We define the spectrum of $x$ to be given by
            \[\sigma(x):=\{\lambda\in\C: \lambda e-x\text{ is not invertible in }{\mathcal{A}}\}.\]
    \end{definition}

    \begin{definition}
        Let ${\mathcal{A}}$ be a commutative unital Banach algebra. We define the spectrum of ${\mathcal{A}}$, written $\sigma({\mathcal{A}})$, to be given by
            \[\sigma({\mathcal{A}}):=\{\psi\in {\mathcal{A}}^\ast:\psi(xy)=\psi(x)\psi(y)\hspace{.15cm}\text{ for each }x,y\in \mathcal{A}\}.\]
    \end{definition}

    \begin{theorem}\label{comphaus}
        Let ${\mathcal{A}}$ be commutative unital Banach algebra. Equipping $\sigma({\mathcal{A}})$ with topology given by pointwise convergence in in ${\mathcal{A}}^\ast$, $\sigma({\mathcal{A}})$ becomes a compact Hausdorff space.
    \end{theorem}

    \begin{proof}
        Notice that for each invertible $x\in {\mathcal{A}}$ and $h\in\sigma({\mathcal{A}})$, we have $h(x)=h(ex)=h(e)h(x)$ whence $h(e)=1$. Consequently, $1=h(xx^{-1})=h(x)h(x^{-1})$ which implies $h(x)\neq 0$. Suppose for some $\lambda\in\C$ we have $|\lambda|>\norm{x}$, then $\lambda e-x$ is invertible by a straightforward geometric series argument, so $\lambda-h(x)=h(\lambda e-x)\neq 0$, and hence $|h(x)|\leq\norm{x}$. This implies that $\sigma({\mathcal{A}})$ is bounded above in the operator norm by 1, and is hence a subset of the closed unit ball in $\mathcal{L}(V,\C)$; moreover, the set is closed under taking pointwise limits in ${\mathcal{A}}^\ast$, so it is compact since it contains its limit points and is bounded (this is of course not true in general, but it is for this weaker topolgy). $\sigma({\mathcal{A}})$ also inherits separability from $\mathcal{L}(V,\C)$.
    \end{proof}

%%%%%%%%%%%%%%%%%%%%%%%%%%%%%%%%%%%%%%%%%%%%%%%%%%%%

We shall now move on to some basic notions from the Gelfand theory. 

    \begin{definition}
    Let ${\mathcal{A}}$ be a commutative unital Banach algebra. We define the Gelfand Transform of ${\mathcal{A}}$ to be the map $\Gamma_{\mathcal{A}}:{\mathcal{A}}\rightarrow C(\sigma({\mathcal{A}}))$ which takes $x\mapsto\widehat{x}$, given by $\widehat{x}(h)=h(x)$.
    \end{definition}

    \begin{theorem}\label{facts}
    Let ${\mathcal{A}}$ be a commutative unital Banach algebra. We have the following facts:
    \begin{enumerate}
        \item $x$ is invertible if and only if $\widehat{x}$ does not have a zero.
        \item $\operatorname{range}({\widehat{x}})=\sigma(x)$.
        \item $\norm{\widehat{x}}_{\text{sup}}=max_{\lambda\in\sigma(x)}|\lambda|$
    \end{enumerate}
    \end{theorem}

    \begin{proof}
        It is clear that $\Gamma_{\mathcal{A}}$ is a homomorphism from ${\mathcal{A}}$ into $C(\sigma({\mathcal{A}}))$, whence the forward direction of $(a)$ follows from the argument in the proof of theorem \ref{comphaus}. For the reverse direction, suppose $x$ is not invertible in ${\mathcal{A}}$. Then the ideal generated by the singleton set $\{x\}$ is proper, and is contained in a maximal ideal. In turn, by Lemma \ref{idealfunct}, there is some $h\in\sigma({\mathcal{A}})$ for which $h(x)=0$, whence the Gelfand transform of $x$ has a zero. (b) follows from (a), in the sense that $\lambda e-x$ is not invertible $\iff$ $\widehat{\lambda e-x}$ has a zero $\iff$ $\lambda = h(x)$ for some $h$. (c) follows readily from fact (b).    
    \end{proof}

We now give an additional structural condition for $\ast$-algebras and show that this coincides with many properties concerning algebraic structure and the Gelfand transform.

    \begin{definition}
    Let ${\mathcal{A}}$ be a commutative unital $\ast$-algebra. Then ${\mathcal{A}}$ is called symmetric if $\Gamma_{\mathcal{A}}$ is a $\ast$-isomorphism in the sense that
    \[\widehat{x^\ast}=\overline{\widehat{x}}.\]
    \end{definition}

    \begin{theorem}\label{c star}
    If ${\mathcal{A}}$ is a commutative unital $\mathbf{C}^\ast$-algebra, then ${\mathcal{A}}$ is symmetric and $\Gamma({\mathcal{A}})$ is dense in $C(\sigma({\mathcal{A}}))$.
    \end{theorem}
    
    \begin{proof}
    Suppose for a moment that for each $x\in {\mathcal{A}}$ we have $x=x^\ast$ implies $\widehat{x}$ is real valued. Decompose each $x\in {\mathcal{A}}$ by setting $x^\ast=a-ib$ with $a^\ast=a$, $b^\ast=b$, and using this decomposition notice $\widehat{x^\ast}=\overline{\widehat{x}}$, whence ${\mathcal{A}}$ is symmetric. We now show that the original assumption follows for a general commutative unital $\mathbf{C}^\ast$-algebra. Fix some $x\in {\mathcal{A}}$ for which $x=x^\ast$, and some $h\in\sigma({\mathcal{A}})$. Write $\widehat{x}(h)=h(x):=\alpha+i\beta$ and set $z=x+ite$ so that $\norm{zz^\ast}=\norm{x}^2+t^2$ and $\norm{h(z)}^2=\alpha^2+(\beta+t)^2$. We have the following:
    \[\begin{split}
        \alpha^2+(\beta+t)^2&=\norm{h(z)}^2\leq\norm{z}^2=\norm{zz^\ast}=\norm{x}^2+t^2\\
        \Rightarrow \norm{x}^2&\geq \alpha^2+\beta^2+2\beta t
    \end{split}\]
    for each $t\in\mathbb{R}$ from which we force $\beta=0$ and $\widehat{x}$ to be real valued. It follows from the argument given at the beginning that ${\mathcal{A}}$ is symmetric. To see that $\Gamma({\mathcal{A}})$ is dense in $C(\sigma({\mathcal{A}}))$, recall that since ${\mathcal{A}}$ is symmetric, the image $\Gamma({\mathcal{A}})$ is closed under involution and observe that it in fact separates points. Apply the Stone-Weierstrass theorem to arrive at the claim.
    \end{proof}

    \begin{theorem}
    Let ${\mathcal{A}}$ be a commutative unital Banach algebra. Then TFAE:
        \begin{enumerate}
            \item $\Gamma_{\mathcal{A}}$ is an isometry between Banach algebras
            \item For each $x\in {\mathcal{A}}$,  $\norm{x^2}=\norm{x}^2$
            \item For each $x\in {\mathcal{A}}$, $\norm{x}=\norm{\widehat{x}}$
        \end{enumerate}
    \end{theorem}

    \begin{proof}
    Suppose $(b)$ holds. Then $\norm{\widehat{x}}=\max{\lambda\in\sigma(x)}|\lambda|=\norm{x}$. Conversely, suppose $(c)$ holds. Then,
    \[\norm{x^2}\leq\norm{x}^2=\norm{\widehat{x}}^2=\norm{\widehat{x}^2}\leq\norm{x^2}\] whence $(b)$ holds. Since $(a)\iff(c)$ is clear, the claim follows.
    \end{proof}

We now present a key result of this section which is instrumental in our second formulation of the finite-dimensional spectral theorem.

    \begin{theorem}[Gelfand-Naimark Theorem]
    Let ${\mathcal{A}}$ be a commutative unital $\mathbf{C}^\ast$-algebra. Then $\Gamma$ is an isometric $\ast$-isomorphism from ${\mathcal{A}}$ into $C(\sigma({\mathcal{A}}))$.
    \end{theorem}
    
    \begin{proof}
    Since ${\mathcal{A}}$ is symmetric, $\Gamma_{\mathcal{A}}$ is a $\ast$-preserving map by \ref{c star}. Notice that for each $x\in {\mathcal{A}}$, we have $\norm{(xx^\ast)^2}=\norm{xx^\ast}^2$ and from the previous result $\norm{\widehat{xx^\ast}}=\norm{xx^\ast}$, hence
    \[\norm{x}^2=\norm{xx^\ast}=\norm{\widehat{xx^\ast}}=\norm{\widehat{x}^2}=\norm{\widehat{x}}^2\]
    since ${\mathcal{A}}$ is symmetric. This shows $\Gamma_{\mathcal{A}}$ is an isometry, from which it follows that $\Gamma_{\mathcal{A}}$ is injective and has closed range.  By \ref{c star}, $\Gamma({\mathcal{A}})$ is also dense in $C(\sigma({\mathcal{A}}))$ so that $\Gamma$ is surjective.
    \end{proof}

Let us now formulate our second version of the Spectral theorem finite-dimensional Hilbert spaces.

    \begin{theorem}[Spectral Theorem II]\label{spectraltheoremii}
    Let $V$ be a finite dimensional Hilbert space, and let $T$ be a normal operator on $V$. Then the following hold:
    \begin{enumerate}
        \item There is an orthonormal basis for $V$ consisting of eigenvectors of $T$.
        \item $T=\sum_{\lambda\in\sigma(T)} \lambda P_{\mathcal{E}_\lambda}$.
    \end{enumerate}
    \end{theorem}

    \begin{proof}
    Let $T\in\mathcal{L}(V)$ be normal, and let ${\mathcal{A}}$ be the $\mathbf{C}^\ast$-subalgebra of $\mathcal{L}(V)$ generated by $\{{\mathcal{A}},{\mathcal{A}}^\ast,I\}$. Then one can check that ${\mathcal{A}}$ is a commutative unital $\mathbf{C}^\ast$-algebra. By the Gelfand-Naimark theorem, ${\mathcal{A}}$ is isomorphic to $C(\sigma({\mathcal{A}}))$. Hence, $\dim\big(C(\sigma({\mathcal{A}}))\big)<\infty$, which in turn forces $\sigma({\mathcal{A}})$ to be finite. Notice then that for each $f\in C(\sigma({\mathcal{A}}))$,
        \[f=\sum_{\lambda\in\sigma({\mathcal{A}})} f(\lambda)\chi_{\lambda}\]
    where $\chi_{\lambda}$ is the characteristic function of $\{\lambda\}$, whence the family $\chi_\lambda$ for $\lambda\in\sigma({\mathcal{A}})$ is a basis for the space $C(\sigma({\mathcal{A}}))$. For each $\chi_\lambda$, $\chi_\lambda^2=\chi_\lambda$, $\chi_\lambda^\ast=\chi_\lambda$, and $\chi_{\cap_\lambda\{\lambda\}}=\Pi_\lambda \chi_\lambda$, from which we see that each $\chi_\lambda$ is actually the image of an orthogonal projection $P_\lambda$ under $\Gamma_{\mathcal{A}}$, and since ${\mathcal{A}}$ is isomorphic to $C(\sigma({\mathcal{A}}))$, that the family $P_\lambda$ is a basis for ${\mathcal{A}}$. In fact, it is easily checked that for each $S\in {\mathcal{A}}$ with Gelfand Transform $\widehat{S}$, one has
        \[S=\sum_{\lambda\in\sigma({\mathcal{A}})} \widehat{S}(\lambda)P_\lambda.\]
    In particular,
       \[T=\sum_{\lambda\in\sigma({\mathcal{A}})} \widehat{T}(\lambda)P_\lambda.\]
    Let us look at the ranges of $P_\lambda$ as subspaces of $V$. Suppose $P_\lambda v=v$. Then, $Tv=TP_\lambda v=P_\lambda Tv=\widehat{T}(\lambda)v$ and we see that $P_\lambda$ is actually a projection onto the eigenspace associated to eigenvalue $\widehat{T}(\lambda)$, whence $(b)$ follows. To see why $(a)$ is true, it suffices to observe that since
    \[I=\sum_{\lambda\in\sigma({\mathcal{A}})}P_\lambda\]
    we have the decomposition 
    \[V=\bigoplus_{\lambda\in\sigma({\mathcal{A}})}\operatorname{Range}(P_\lambda)=\bigoplus_{\mu\in\sigma(T)}\mathcal{E}_\mu.\]
    Each eigenspace associated to $T$ admits an orthonormal basis, and by taking the union all eigenspace bases, one arrives at such an orthonormal basis for the whole space.
    \end{proof}

\section{Fourier Analysis of Finite Abelian Groups}

In this section, we strive to apply the spectral theory developed in the preceding sections to the special setting of Fourier analysis on finite abelian groups; in particular, with the goal of proving Plancherel Formula for finite abelian groups, and classifying the to a general finite abelian group in mind. We follow the classic exposition of Stein, Shakarchi\citep[ch. 7]{St07}.

    \begin{definition}
    Let ${G}$ be a set closed under a binary operation $\cdot:{G}\times {G}\rightarrow {G}$. We say $({G},\cdot)$ is an abelian group if it satisfies the following four axioms:
        \begin{enumerate}
            \item Associativity: For each $a,b,c\in {G}$, $a\cdot(b\cdot c)=(a\cdot b)\cdot c$.
            \item Identity: There is some $e\in {G}$ so that for each $a\in {G}$,  $a\cdot e=e\cdot a=a$.
            \item Inverses: For each $a\in {G}$ there is $a^{-1}$ so that $a\cdot a^{-1}=a^{-1}\cdot a=e$.
            \item Commutativity: For each $a,b\in {G}$ we have $a\cdot b=b\cdot a$.
        \end{enumerate}
    We will omit the $\cdot$ in general, and use juxtaposition to indicate that the group operation is being performed.
    \end{definition}

Two examples of abelian groups include $(\mathbb{R},+)$ and $(\mathbb{R}_{>0},\times)$. A key example of an abelian group is the unit circle $S^1:=\{z\in\mathbb{C}:|z|=1\}$ under multiplication, where for each $\omega\in S^1$, $\omega^{-1}=\overline{\omega}$.\par
In some sense, we seek to generalize to the setting of abelian groups the notion of a multiplicative functional and consequently the notion of spectrum. We do not necessarily have the additive or multiplicative structure of, say, a Banach algebra, but we do have the notion of a homomorphism, which allows us a path forward.
    \begin{definition}
    Let $({G},\cdot)$, $({H},\times)$ be two abelian groups. A homomorphism is a map $\phi:{G}\rightarrow {H}$ which preserves the group operations, in the sense that for each $a,b\in {G}$
        \[\phi(a\cdot b)=\phi(a)\times \phi(b).\]
    If $\phi$ is a bijection, we say it is an isomorphism and that the groups ${G},{H}$ are isomorphic and we write ${G}\cong {H}$.
    \end{definition}

    \begin{definition}
    A character on an abelian group $({G},\cdot)$ is a homomorphism $\zeta:{G}\rightarrow S^1$.
    \end{definition}
    
    \begin{definition}
      Let $({G},\cdot)$ be an abelian group. The dual group of ${G}$ is the family $$\widehat{{G}}:=\{\zeta:{G}\rightarrow S^1:\zeta(ab)=\zeta(a)\zeta(b)\}$$ of characters on ${G}$, under operation of pointwise multiplication with identity given by $e\equiv 1$.
    \end{definition}

It turns out that characters are in fact the analogues of multiplicative functionals which we seek. This is illustrated by the following facts.

    \begin{lemma}\label{charlemma}
    Let $({G},\cdot)$ be a finite abelian group and suppose $f:{G}\rightarrow\mathbb{C}\backslash\{0\}$ be a multiplicative function. Then $f$ is a character.
    \end{lemma}
    
    \begin{proof}
    Since ${G}$ has finitely many elements, $|f(a)|$ is bounded for $a\in {G}$. Since $|f(a^n)|=|f(a)|^n$ for each $a\in {G}$, $n\in\mathbb{N}$, it follows that $|f(a)|=1$ for each $a$.
    \end{proof}

    \begin{definition}
    Let $({G},\cdot)$ be a finite abelian group. Define the function space $\ell_2({G}):=\{f:{G}\rightarrow\mathbb{C}\}$, and equip it with Hermitian inner product 
        \[\langle f,g\rangle =\frac{1}{|{G}|}\sum_{a\in {G}}f(a)\overline{g(a)}\]
    where $|{G}|$ is the number of elements in ${G}$, and with respect to which $\ell_2({G})$ is a Hilbert space isomorphic to $\mathbb{C}^{|{G}|}$ as a vector space with equivalent (not quite isometric) norms as Hilbert spaces.
    \end{definition}
    
    \begin{theorem}\label{charorthfam}
    Let $({G},\cdot)$ be a finite abelian group. Then the dual group $\widehat{{G}}$ is an orthonormal family in $\ell_2({G})$.
    \end{theorem}

    \begin{proof}
    To see that each $\zeta\in\widehat{{G}}$ is unit length, notice that for each $\zeta\in\widehat{{G}}$ one has
        \[\norm{\zeta}^2=\frac{1}{|{G}|}\sum_{a}|\zeta(a)|^2=1.\]
    To see why $\widehat{{G}}$ is an orthogonal collection, suppose we have two distinct characters $\zeta,\omega\in\widehat{{G}}$. Find some $b\in {G}$ for which $(\zeta\omega^{-1})(b)\neq 1$. Then,
        \[(\zeta\omega^{-1})(b)\sum_{a\in {G}}(\zeta\omega^{-1})(a)=\sum_{a\in\Gamma}(\zeta\omega^{-1})(ba)=\sum_{a\in\Gamma}(\zeta\omega^{-1})(a)\Rightarrow \sum_{a\in\Gamma}(\zeta\omega^{-1})(a)=0\]
    whence $\frac{1}{|{G}|}\sum_a \zeta(a)\overline{\omega}(a)=0$ and $\zeta,\omega$ are orthogonal.
    \end{proof}

    \begin{theorem}\label{charbasis}
    Let $({G},\cdot)$ be a finite abelian group. The family $\widehat{{G}}$ forms an orthonormal basis for the function space $\ell_2({G})$.
    \end{theorem}

    \begin{proof}
    Consider the operator $T_a\in\mathcal{L}(\ell_2({G}))$ defined by
        \[(T_af)(x)=f(a x).\]
    Let us generate an algebra with this family of operators. One may easily check that for each $a\in {G}$, the adjoint $T_a^\ast$ coincides with $T_{a^{-1}}$, so the subalgebra $\mathcal{T}$ of $\mathcal{L}(\ell_2({G}))$ generated by $\{T_a\}_{a\in {G}}$ is closed under adjoint, and since $\mathcal{T}$ is abelian, it is a commutative unital $\mathbf{C}^\ast$-algebra under the usual operations, with identity given by the identity operator $T_e=I$ on $\ell_2({G})$.
    
    By the Gelfand-Naimark Theorem, $\mathcal{T}$ is isomorphic to $C(\sigma(\mathcal{T}))$. It follows that $\dim{\mathcal{T}}=|{G}|=\dim{C(\sigma(\mathcal{T}))}=|\sigma(\mathcal{T})|$ and that each $f\in C(\sigma(\mathcal{T}))$ may be written
        \[f=\sum_{\lambda\in\sigma(\mathcal{T})}f(\lambda)\chi_{H}\]
    where $\chi_{H}$ is the characteristic function of $\{\lambda\}\subset\sigma(\mathcal{T})$. It then follows that $\{\chi_\lambda\}_{H}$ is a basis for $C(\sigma(\mathcal{T}))$, and by the arguments in the proof of \eqref{spectraltheoremii}, that each $T_a\in\mathcal{T}$ may be written
        \begin{equation}\label{Tadecom}
        T_a=\sum_{\lambda\in\sigma(\mathcal{T})}\widehat{T_a}(\lambda)P_\lambda \end{equation}
    where $\widehat{T_a}$ is the Gelfand transform of $T_a$, and $P_\lambda=\{\Gamma^{-1}(\chi_{\{\lambda\}})$ are some orthogonal projections on $\ell_2({G})$ yet to be fully determined. Indeed, let us examine their ranges. Suppose $P_\lambda f=f$ for some $f\in\ell_2({G})$. Then,
        \[T_a f=T_a P_\lambda f=\widehat{T_a}(\lambda) f\]
    whence $f$ is a eigenvector associated to $T_a$ with eigenvalue $\widehat{T_a}(\lambda)$. In fact, the conclusion on $f$ is even stronger. $f$ is a simultaneous eigenvector for \textit{every} $T_a\in\mathcal{T}$ with eigenvalues $\widehat{T_a}(\lambda)$. It follows that $P_\lambda$ is a projection onto the space 
        \[E_\lambda:=\{f\in\ell_2(g):T_a f=\widehat{T_a}(\lambda)f\hspace{0.2cm}\forall a\in {G}\}.\]
    and, from applying \eqref{Tadecom} to $T_e$, we obtain the decomposition
        \begin{equation}\label{decompofell}
        \ell_2({G})=\bigoplus_{\lambda\in\sigma(\mathcal{T})}E_\lambda.
        \end{equation}
    For each $\lambda\in\sigma(\mathcal{T})$, suggestively construct a function $\zeta_\lambda$ on ${G}$ defined by
        \begin{equation}\label{chilambda}
            \zeta_\lambda(a)=\widehat{T_a}(\lambda).
        \end{equation}
    We check that for each $a,b\in {G}$ we have
        \[\begin{split}
        \zeta_\lambda(ab)&=\widehat{T_{ab}}(\lambda)=\lambda(T_{ab})=\lambda(T_aT_b)\\&=\lambda(T_a)\lambda(T_b)=\widehat{T_a}(\lambda)\widehat{T_b}(\lambda)=\zeta_\lambda(a)\zeta_\lambda(b)
        \end{split}\]
    and that since each $T_a$ is invertible, $\widehat{T_a}$ is never zero, and consequently that $\zeta_\lambda$ is also never zero. By \eqref{charlemma}, it follows that $\zeta_\lambda$ is in fact a character on ${G}$, and by construction, that $\zeta_\lambda\in E_\lambda$. One also observes from our definition in \eqref{chilambda} that $\zeta_\lambda$, when considered as a functional $T_a\mapsto \widehat{T_a}(\lambda)$ on $\mathcal{T}$, is also an element of the spectrum of $\mathcal{T}$, whence the family $\{\zeta_\lambda\}_{\lambda\in\sigma(\mathcal{T})}$ in in fact simply $\sigma(\mathcal{T})$. It also follows from \eqref{decompofell}, \eqref{charorthfam}, and our construction in \eqref{chilambda}, that the family $\{\zeta_\lambda\}_{\lambda\in\sigma(\mathcal{T})}$ is an orthonormal basis for $\ell_2({G})$ and exhausts every character on ${G}$ (to see the second fact, suppose there was a different character and one contradicts \eqref{decompofell}). We conclude that $\sigma(\mathcal{T})=\widehat{{G}}=\{\zeta_\lambda\}_{\lambda\in\sigma(\mathcal{T})}$ is an orthonormal basis of $\ell_2({G})$ as claimed.
    \end{proof}

We now present the classic Plancherel Formula in the setting of abelian groups which follows immediately from Theorem \eqref{charbasis}.
    \begin{theorem}[Plancherel Formula]
    Let $({G},\cdot)$ be a finite abelian group. For each $f\in\ell_2({G})$ and $\zeta\in\widehat{{G}}$, set
        \[\widehat{f}(\zeta)=\ip{f,\zeta}=\frac{1}{|{G}|}\sum_{a\in {G}}f(a)\overline{\zeta(a)}.\]
    Then
        \[\norm{f}^2=\sum_{\zeta\in\widehat{{G}}}|\widehat{f}(\zeta)|^2.\]
    \end{theorem}

Let us now transition to the goal of classifying the duals of general finite abelian groups so as to obtain a concrete understanding of what these functions look like in the wild. We shall first recall some group-theoretic definitions and develop a result about the dual of a product group.
    \begin{definition}
    Let $({G},\cdot)$ and $({H},\circ)$ be two abelian groups. We define the product group ${G}\times {H}$ by the Cartesian product ${G}\times {H}$ with operation $\otimes$ defined by
        \[(g,h)\otimes (g',h'):=(g\cdot g',h\circ h')\]
    for each $g\in {G}$, $h\in {H}$.
    \end{definition}
We leave the verification that this is actually a group as an exercise; identity and inverse elements in the product are given by pairs of the respective elements in the original groups. Since it is clear that ${G},{H}$ abelian implies ${G}\times {H}$ is also abelian, we may construct a dual product group as follows:
    \begin{definition}
    Let ${G},{H}$ be abelian groups and let ${G}\times {H}$ be the product group as defined above. The dual group of the product, written $\widehat{{G}\times {H}}$, is the group defined by the set of characters on the product under pointwise multiplication, written
        \[\widehat{{G}\times {H}}:=\{\zeta:{G}\times {H}\rightarrow S^1\big{|} \zeta((g_1,h_1)\otimes(g_2,h_2))=\zeta(g_1,h_1) \zeta(g_2,h_2)\}.\]
    for every $g_i,h_i\in {G},{H}$ resp., where for each $\zeta,\omega\in\widehat{{G}\times {H}}$, $(\zeta\omega)(a,\alpha)=\zeta(a,\alpha)\omega(a,\alpha)$.
    \end{definition}
We now prove a nice result regarding such product groups.
    \begin{theorem}\label{proddual}
    Let ${G},{H}$ be finite abelian groups. Then we have the relation
        \[\widehat{{G}\times {H}}\cong \widehat{{G}}\times\widehat{{H}}.\]
    \end{theorem}
    
    \begin{proof}
    To prove the claim we must produce a bijective homomorphism between the two groups. Let us construct such a function as follows. Let
        \[\psi:\widehat{{G}\times {H}}\rightarrow \widehat{{G}}\times\widehat{{H}}:\zeta\mapsto (\zeta(\cdot,e_{H}),\zeta(e_{G},\cdot))\]
    where $e_{H}$ is the identity in ${H}$ and $e_{G}$ is the identity in ${G}$. We first show that this function is a homomorphism, in the sense that for any two characters $\zeta,\omega$ on the product ${G}\times {H}$, $\psi(\zeta\omega)=\psi(\zeta)\otimes\psi(\omega)$ where $\otimes$ refers to the operation on the product of the dual groups $\widehat{{G}}\times\widehat{{H}}$. Indeed, this is verified as follows:
        \[\begin{split}
            \psi(\zeta\omega)&= (\zeta\omega(\cdot,e_{H}),\zeta\omega(e_{G},\cdot))\\
            &=(\zeta(\cdot,e_{H}),\zeta(e_{G},\cdot))\otimes(\omega(\cdot,e_{H}),\omega(e_{G},\cdot))=\psi(\zeta)\otimes\psi(\omega)
        \end{split}\]
    To check surjectivity, we show that any element $(\zeta,\omega)$ in the product of the duals $\widehat{{G}}\times\widehat{{H}}$ may be realized as the image of a character on the product group ${G}\times {H}$. Indeed, by choosing $\rho(g,h)=\zeta(g)\omega(h)$, which is a character on the product group, for each $g\in {G},h\in {H}$, we find that
        \[\psi(\rho)=(\rho(\cdot,e_{H}),\rho(e_g,\cdot))=(\zeta,\omega)\]
    so that $\psi$ is a surjection. The verification that $\psi$ is injective is left to the reader. 
    \end{proof}
Let us now look at some specific finite abelian groups and their duals, and then recall a classification theorem from algebra.
    \begin{definition}
    Let $N\geq 1$ be an integer. We define the cyclic group of order $N$ to be the set $\mathbb{Z}_N:=\{n\in\mathbb{Z}:0\leq n\leq N-1\}$ under the operation of addition modulo $N$.
    \end{definition}
    \begin{theorem}
    For each $n\in\mathbb{Z}_N$, define a character on $\mathbb{Z}_N$ by
        \[e_n(k)=e^{\frac{2\pi ink}{N}}\]
    Then the dual to $\mathbb{Z}_N$ is given by
        \[\widehat{\mathbb{Z}_N}=\{e_n\}_{n\in\mathbb{Z}_N}.\]
    \end{theorem}
    \begin{proof}
    Let us first show that each $e_n$ is a character. Let $N\geq 1$, $e_n$ be fixed, and let $a,b\in\mathbb{Z}_N$. Then, by elementary properties of complex exponentials, one has $e_n(a+b)=e_n(a)e_n(b)$ whence $e_n$ is a homomorphism $\mathbb{Z}_N\rightarrow S^1$ and a character. Since $|\{e_n\}_n|=|\mathbb{Z}_N|$, we have found every character on $\mathbb{Z}_N$ by Theorem \eqref{charbasis}.
    \end{proof}
Let us now recall a classification theorem for finite abelian groups from algebra and apply it here. We state a version from \cite[13.3]{Ju97}.

    \begin{theorem}[Classification of Finite Abelian Groups]\label{classgrps}
    Let ${G}$ be a finite abelian group. Then ${G}$ is isomorphic to the product of finitely many cyclic groups, each of order $p^k$ for some prime $p$, and $k\geq 1$. Symbolically,
        \[{G}\cong\prod_{i=1}^{n}\mathbb{Z}_{p_i^{k_i}}.\]
    \end{theorem}

We may now state the desired description of the dual to an arbitrary finite abelian group:
    
    \begin{theorem}[Classification of Dual Groups]
    Let ${G}$ be a finite abelian group. Then the dual $\widehat{{G}}$ is isomorphic to the product of finitely many duals of finite cyclic groups of prime-power order. In other words, we have
        \[\widehat{{G}}\cong\widehat{\prod_{i=1}^{n}\mathbb{Z}_{p_i^{k_i}}}\cong\prod_{i=1}^{n}\widehat{\mathbb{Z}_{p_i^{k_i}}}.\]
    \end{theorem}
    
    \begin{proof}
    First apply result \eqref{classgrps} to ${G}$, then apply \eqref{proddual} and the claim follows.
    \end{proof}
    
\section{Fourier Analysis on Unit Circle}

In this section, we will derive some classical theorems from Fourier Analysis by applying the spectral theory we have developed to $L^2(S^1,\mu)$. Let us begin this exploration by first looking at $S^1$ as a domain itself. We shall think of $S^1$ as an Abelian group under the operation of multiplication, and as a compact Hausdorff topological subspace of $\mathbb{C}$ equipped with the subspace topology. As a consequence, there exists a unique (up to multiplication by positive constant) normalized left- and right-invariant inner- and outer-regular measure $\mu$ on the Borel sets associated to the topology on $S^1$.\par 
More explicity, the measure $\mu$ is a non-negative function on the $\sigma$-algebra generated by the open sets in $S^1$, satisfying the following conditions:
    \begin{enumerate}[label=(\roman*)]
        \item For each countable collection of disjoint Borel sets $F_i$, $i\in\mathbb{N}$, we have $\mu(\cup_i F_i)=\sum_i\mu(F_i).$ \textit{(countable additivity)}
        \item $\mu(S^1)=1$ and $\mu(\varnothing)=0$. \textit{(normality)}
        \item For every Borel set $A$, we have $\mu(A)=\inf{\{\mu(U):U\supset A, U\text{ open}\}}.$ \textit{(outer-regularity)}
        \item For every open set $B$, we have $\mu(B)=\sup{\{\mu(K):K\subset B, K \text{ compact}\}}$.\\ \textit{(inner-regularity)}
        \item For every Borel set $F$ and every $\omega\in S^1$, and cosets $\omega F:=\{\omega f:f\in F\}$, $F\omega :=\{f\omega :f\in F\}$, we have $\mu(\omega F)=\mu(F \omega)=\mu(F)$. \textit{(left- and right-invariance)}
    \end{enumerate}
We shall call $\mu$ the normalized Haar measure on $S^1$, and we may then define an integral $\int_{S^1}d\mu$ for Borel-measurable functions on the unit circle $S^1$ through the usual techniques of taking the limit of integrals of simple functions. For more information on this technique, the reader is referred to a Real Analysis text such as \cite{Ro10}, or for a brief treatment on Haar measure, the paper \cite{Gl10}. We mention the inner- and outer-regularity for thoroughness, but we shall not need them explicitly here.\par
Let us now define our function space of interest:
    \begin{definition}
    We define the space of square-integrable functions on the unit circle:
        \[L^2(S^1,\mu):=\{f:S^1\rightarrow\mathbb{C}\big{|} f\text{ Borel measurable },\int_{S^1}|f(\omega)|^2 d\mu (\omega)<\infty\}\]
    which we equip with inner product
        \[\ip{f,g}=\int_{S^1}f(\omega)\overline{g(\omega)}d\mu(\omega)\]
    and associated norm $\norm{\cdot}_2$ given by
        \[\norm{f}_2^2=\int_{S^1}|f(\omega)|^2d\mu(\omega)\]
    \end{definition}
$L^2(S^1,\mu)$ is complete, i.e., a Hilbert space. At this point, we may pose a few questions about this function space regarding the extent to which the results of the previous section apply. Does the dual group $\widehat{S^1}$ form an orthonormal basis for this space, as it did for a similar space on finite Abelian groups? Does $L^2$ admit a decomposition of the form \eqref{decompofell}? Do we have some sort of analogue of the Plancherel identity here?\par
Inherent is the identification of the dual group to $S^1$. Since this Abelian group is no longer finite, we need to alter our definition slightly. In particular,
    \begin{definition}
    We define the dual group to $S^1$ to be the set
        \[\widehat{S^1}:=\{ f:S^1\rightarrow \mathbb{C}\backslash\{0\}\hspace{0.2cm}\big{|}\hspace{0.2cm}f\text{ continuous, homomorphism }\}\]
    under the operation of pointwise multiplication.
    \end{definition}
Since $S^1$ is a compact subset of $\mathbb{C}$, and each $\zeta\in\widehat{S^1}$ is continuous, it follows that $|\zeta|\equiv 1$ by an argument identical to the one provided in the previous section. Let us introduce the following family of continuous functions $S^1\rightarrow S^1$.
    \begin{definition}
    For each $n\in\mathbb{Z}$ let us define
        \[e_n:S^1\rightarrow S^1:\omega\mapsto\omega^n.\]
    \end{definition}
If we identify each $\omega\in S^1$ with some complex exponential $e^{i\theta}$ for some $\theta\in[0,2\pi)$, then $e_n(\theta)=e^{2\pi i n\theta}$, which we see resembles our definition of $e_n$ earlier and justifies its notation. We then have the expected identification
    \[\widehat{S^1}=\{e_n\}_{n\in\mathbb{Z}}.\]
For a proof of this fact, see \cite{Fo15}[p. 98]. Let us continue our analysis of the dual $\widehat{S^1}$ by checking that it is an orthonormal family in $L^2(S^1,\mu)$.
    \begin{lemma}[Orthogonality Relations of $\widehat{S^1}$ in $L^2(S^1,\mu)$]
    The family $\widehat{S^1}=\{e_n\}_n$ is orthonormal in $L^2(S^1,\mu)$.
    \end{lemma}

    \begin{proof}
    To check orthogonality, let us observe that for each $n\in\mathbb{Z}$, $\xi\in S^1$, we have the identity
        \[\int_{S^1}e_n(\omega)d\mu(\omega)=\int_{S^1}e_n(\xi\omega)d\mu(\omega)=\int_{S^1}e_n(\xi)e_n(\omega)d\mu(\omega)=e_n(\xi)\int_{S^1}e_n(\omega)d\mu(\omega)\]
    whence $\int_{S^1}e_n d\mu=0$ for each $e_n\neq 1\iff n\neq 0$. One checks that $\ip{e_n,e_m}=\int_{S^1}e_{n-m}d\mu$ whence $\ip{e_n,e_m}=0$ for each $n\neq m$. We next compute the norm
        \[\norm{e_n}^2=\int_{S^1}|e_n(\omega)|^2d\mu(\omega)=\mu(\omega)=1\]
    and the proof is complete.
    \end{proof}
As we did in the proof of the fact that $\widehat{G}$ is an orthonormal basis for $\ell(G)$, let us now introduce a helpful family of operators on $L^2(S^1,\mu)$ which will provide a vehicle for many useful arguments. Namely, let $\mathcal{S}$ be the subalgebra of $\mathcal{L}(L^2(S^1,\mu))$ generated by the family of operators $\{T_\omega\}_{\omega\in S^1}$ defined by
    \[(T_\omega f)(\xi)=f(\omega\xi),\hspace{.5cm}\xi\in S^1,\]
for each $f\in L^2$. It is clear that this algebra is commutative since $S^1$ is commutative, and one can check that for each $\omega$, the adjoint $T_\omega^\ast$ is given by $T_{\omega^{-1}}$, whence with involution $T_\omega\mapsto T_{\omega^{-1}}$, $\mathcal{S}$ is a commutative unital $\mathbf{C}^\ast$-algebra. Recall that the spectrum of this algebra $\sigma(\mathcal{S})$ is the family of continuous nonzero multiplicative functionals on the algebra, and with the operation of pointwise multiplication, in this setting, $\sigma(\mathcal{S})$ becomes an Abelian group. This begins to sound like the dual group, and in fact, we have the following key relation.
    \begin{theorem}
    We have the relation
        \[\sigma(\mathcal{S})\cong \widehat{S^1}\]
    by a group isomorphism.
    \end{theorem}

    \begin{proof}
    We shall prove this result by constructing the isomorphism explicitly. Let us fix some $h\in\sigma(\mathcal{S})$ and construct an associated character $\zeta_h$ on $S^1$ by setting
        \[\zeta_h(\omega)=h(T_\omega)\]
    for each $\omega\in S^1$. Since $h$ is multiplicative and continuous, $\zeta_h$ is a character. Let $\Delta$ be the mapping $\sigma(\mathcal{S})\rightarrow\widehat{S^1}$ which takes $h\mapsto\zeta_h$. We shall show $\Delta$ is an isomorphism. To check that it is a homomorphism, we verify that for each $h,g\in\sigma(\mathcal{S})$, $\omega\in S^1$,
        \[(\Delta hg)(\omega)=(hg)(T_\omega)=h(T_\omega)g(T_\omega)=(\Delta h)(\omega)(\Delta g)(\omega)\]
    whence $\Delta$ is a homomorphism. Let us now check injectivity and surjectivity of this mapping. For surjectivity, if we start with some character $\zeta\in\widehat{S^1}$, we may construct a multiplicative functional $h_\zeta$ on $\mathcal{S}$ defined by
        \[h_\zeta(T_\omega)=h(\omega).\]
    Check that $h_\zeta\in\sigma(\mathcal{S})$ and $\Delta(h_\zeta)=\zeta$, whence $\Delta$ is surjective. To check injectivity, suppose we have two characters $\zeta_h,\zeta_g$ in the range of $\Delta$. Then $\zeta_h(\omega)\zeta_g(\omega^{-1})\equiv 1$, whence $h(\omega)g(\omega^{-1})\equiv 1$ and $h\equiv g$. This completes the proof.
    \end{proof}
    
Using the isomorphism relation above and the mapping $\Delta$ as introduced in the proof, we may give an explicit decomposition of the spectrum $\sigma(\mathcal{S})$ as follows.

    \begin{theorem}
    For each $n\in\mathbb{Z}$ define a functional $h_n$ on $\mathcal{S}$ by 
        \[h_n(T_\omega)=e_n(\omega).\]
    Then we have
        \[\sigma(\mathcal{S})=\{h_n\}_{n\in\mathbb{Z}}.\]
    \end{theorem}

    \begin{proof}
     Since $\{h_n\}\subset\sigma(\mathcal{S})$ we only need show any $h\in\sigma(\mathcal{S})$ may be realized as some $h_n$. It suffices to check here that $\Delta(h_n)=e_n$. This is trivial based  on how we constructed $h_n$, whence the equality follows from the following argument: suppose we have some $h\in\sigma(\mathcal{S})$, then $\Delta(h)=e_n$ for some $n\in\mathbb{Z}$ since $\Delta(h)\in\widehat{S^1}\subset\{e_n\}_{n\in\mathbb{Z}}$, and by going backwards, $\Delta^{-1}(e_n)=h_n=h$, whence $h$ is equal to $h_n$, since $\Delta$ acts as a bijection.
    \end{proof}
    
Let us now set up a decomposition of $L^2(S^1,\mu)$ in two stages. Our first formulation will be a resolution of the identity.

    \begin{theorem}[Resolution of $I$ on $L^2(S^1,\mu)$]\label{residl2}
    There exists a sequence of mutually orthogonal projections $P_n$, $n\in\mathbb{Z}$, so that the identity operator $I$ on $L^2(S^1,\mu)$ has the expression
        \[I=\sum_{n\in\mathbb{Z}}P_n.\]
    \end{theorem}
    \begin{proof}
    Recalling that $\mathcal{S}$ is a commutative unital $\mathbf{C}^\ast$-algebra, we invoke the Gelfand-Naimark theorem to observe that $\mathcal{S}\cong C(\sigma(\mathcal{S}))$ by isometric $\ast$-isomorphism. Noting that $\sigma(\mathcal{S})=\{h_n\}$, we have for each $f\in C(\sigma(\mathcal{S}))$ the expression
        \begin{equation}\label{contfunc}
        f=\sum_{n\in\mathbb{Z}}f(h_n)\chi_{\{h_n\}}
        \end{equation}
    for characteristic functions $\chi_{\{h_n\}}$. Recalling here that $\chi_2=\chi$, $\overline{\chi}={\chi}$, it follows that each characteristic function $\chi_{\{h_n\}}$ has as its preimage under the Gelfand transform an orthogonal projection $P_n$. In particular, for each operator $T_\omega$ on $L^2$, one verifies the expression
        \[T_\omega=\sum_{n\in\mathbb{Z}}\widehat{T}(h_n)P_{n}\]
    Applying the formulation \eqref{contfunc} to the constant function $1$ and its preimage in $\mathcal{S}$, the identity operator, we have the desired expression
        \[I=\sum_{n\in\mathbb{Z}}P_n.\]
    \end{proof}
We now specify the ranges of the projections and explain the role of $\widehat{S^1}$.
    \begin{theorem}[Decomposition of $L^2(S^1,\mu)$]
    Recalling the notation in \eqref{residl2}, we have the identification
        \[\operatorname{Range}(P_n)=\operatorname{span}\{e_n\}:=E_n\]
    from which it follows that the family $\widehat{S^1}$ is an orthonormal basis for $L^2(S^1,\mu)$, and we have the decomposition
        \[L^2(S^1,\mu)=\bigoplus_{n\in\mathbb{Z}} E_n.\]
    \end{theorem}
    \begin{proof}
    Fix $n\in\mathbb{Z}$ and suppose for some $f\in L^2(S^1,\mu)$ we have $P_n f=f$. This implies that for each $\omega$, $T_\omega f=T_\omega P_n f=\widehat{T_\omega}(h_n) f$ whence $f$ is a simultaneous eigenvector for each $T_\omega$ with eigenvalues $\widehat{T_\omega}(h_n)$. Let us compute these eigenvalues: $\widehat{T_\omega}(h_n)=h_n(T_\omega)=e_n(\omega)$ by definition of $h_n$, and we have that for each $\omega\in S^1$, $T_\omega f=e_n(\omega) f$ which implies the following
        \[f(\omega)=(T_\omega f)(1)=e_n(\omega)f(1)\]
    i.e. $f\in \operatorname{span}(e_n)=E_n$. The claim follows.
    \end{proof}
To demonstrate the importance and strength of the results we have shown, we now read off four results from the Fourier analysis of the unit circle which follow immediately from the spectral decompositions we have developed.    
    \begin{theorem}
    For each $f\in L^2(S^1,\mu)$ we have the Fourier series representation
        \[f(\omega)=\sum_{n=-\infty}^{+\infty}c_n e_n(\omega)\]
    of $f$ as a series of the family of functions $e_n$ converging in the mean-square sense, in that
        \[\lim_{N\rightarrow\infty}\int_{S^1}\big{|}f(\omega)-\sum_{n=-N}^{N}c_ne_n(\omega)\big{|}^2d\mu(\omega)=0\]
    with coefficients $c_n\in\mathbb{C}$ given by 
        \[c_n=\ip{e_n,f}=\int_{S^1}e_n(\omega)\overline{f(\omega)}d\mu(\omega)\]
    and $L^2$-norm given by the Plancherel identity
        \[\norm{f}^2=\sum_{n=-\infty}^{+\infty}|c_n|^2.\]
    \end{theorem}

    \begin{proof}
    The first and third equalities follow from the fact that $\{e_n\}$ is an orthonormal basis for $L^2(S^1,\mu)$, the second is an explicit statement of convergence in $L^2$, and the fourth is a standard norm computation done in any inner product space.
    \end{proof}

The preceding result is important because periodic square-integrable functions in the space $L^2[0,2\pi)$ can be isometrically embedded in the space $L^2(S^1,\mu)$, so that the conclusions above on $L^2(S^1,\mu)$ translate immediately over to conclusions on $L^2[0,2\pi)$.

\newpage
\bibliography{references}
\bibliographystyle{plainnat}

\end{document}